\documentclass{amsart}

\usepackage{amsmath}
\usepackage{amssymb}
\usepackage{mathrsfs}
\usepackage{amsthm}
\usepackage{comment}
\usepackage{color}

\title[A little more on $ \Gamma(R) $ and $ \AG(R) $]{A little more on the zero-divisor graph and the annihilating-ideal graph of a reduced ring}

\author{Mehdi Badie}


 
\theoremstyle{plain}
\newtheorem{Thm}{Theorem}[section]
\newtheorem{Lem}[Thm]{Lemma}

\newtheorem{Pro}[Thm]{Proposition}
\newtheorem{Cor}[Thm]{Corollary}

\newcommand{\AG}{\mathbb{AG}}
\newcommand{\AC}{\A{X}}
\newcommand{\AGC}{\mathbb{AG}(X)}
\newcommand{\A}[1]{\mathbb{A}(#1)^*}
\newcommand{\ff}{if and only if }
\newcommand{\Mi}{\mathrm{Min}}

\newcommand{\An}{\mathrm{Ann}}

\newcommand{\Sp}{\mathrm{Spec}}

\newcommand{\Ec}{\mathrm{ecc}}

\newcommand{\Ra}{\mathrm{Rad}}

\newcommand{\B}{\mathcal{B}}

\newcommand{\gi}{\mathrm{gi}}
\newcommand{\dt}{\mathrm{dt}}
\newcommand{\Ge}[1]{\big< #1 \big>}
\newcommand{\sz}{{sz^\circ}}

\newcommand{\R}{\mathbb{R}}
\newcommand{\N}{\mathbb{N}}
\newcommand{\Z}{\mathbb{Z}}

\begin{document}

\begin{abstract}
We have tried to translate some graph properties of $ \AG(R) $ and $ \Gamma(R) $ to the topological properties of Zariski topology. We prove that $ \Ra(\Gamma(R))  $ and $ \Ra(\AG(R)) $ are equal and they are equal to 3, \ff the zero ideal of $ R $ is an anti fixed-place ideal, \ff $ \Mi(R) $ does not have any isolated point, \ff $ \Gamma(R) $ is triangulated, \ff $ \AG(R) $ is triangulated. Also, we show that if the zero ideal of a ring $ R $ is a fixed-place ideal, then $ \dt_t(\AG(R)) = |\B(R)| $ and also if in addition $ |\Mi(R)| > 2 $, then $ \dt(\AG(R)) = |\B(R)| $. Finally, it has been shown that $ \dt(\AG(R)) $ is finite, \ff $ \dt_t(\AG(R) $ is finite; \ff $ \Mi(R) $ is finite. 
\end{abstract}

\maketitle

\section{Introduction}

Let $ R $ be a commutative ring with unity. By $\Sp(R)$  we mean the set of all prime ideals of $R$. A‌ \emph{semi-prime ideal} means an ideal which is an intersection of prime ideals. $ R $ is called a \emph{reduced ring}, if the zero ideal of $ R $ is semi-prime. Through this paper $ R $ is the \emph{commutative unitary reduced ring}. For each ideal $ I $ of $ R $ and each subset $ S $ of $ R $, we denote the ideal $ \{ x \in R: Sx \subseteq I\}$ by $(I:S)$. When $I = \{ 0 \}$ we write $ \An(S) $ instead of $(\{0\}:S)$ and call it the \emph{annihilator} of $S$. Also we write $ \An(a) $ instead of $ \An(\{a\}) $. A prime ideal $ P $ is said to be a \emph{minimal prime ideal} over an ideal $ I $ if there are not any prime ideal strictly contained in $ P $ that contains $ I $. By $ \Mi(I) $ we mean the set of all minimal prime ideals over $ I $; We use $ \Mi(R) $ instead of $ \Mi(\{0\}) $.  A prime ideal $ P $ is called a \emph{Bourbaki associated prime divisor} of an ideal $ I $ if $ (I‌: x)=P$, for some $ x \in R $. We denote the set of all Bourbaki associated prime divisors of an ideal $ I $ by $ \mathcal{B}(I) $. It is easy to see that $\mathcal{B}(I) \subseteq \Mi(I)$, for any ideal $I$ of a ring $R$. We use $\mathcal{B}(R)$ instead of $ \mathcal{B}(\{0\}) $.  Let $ I $ be a semi-prime ideal, $ P_\circ \in \Mi(I) $ is called \emph{irredundant} with respect to $ I $  if $ I \neq \bigcap_{ P_\circ \neq P \in \Mi(I) } P$. If $ I $ is equal to the intersection of all irredundant ideals with respect to $ I $, then we call it a \emph{fixed-place ideal}, exactly, by \cite[Theorem 2.1]{aliabad2013fixed}, we have $ I = \bigcap \mathcal{B}(I) $. If $ \mathcal{B}(I) = \emptyset $, then $ I $ is called an \emph{anti-fixed place ideal}. We use $ \mathcal{B}(R) $ instead of $ \mathcal{B}(\{0\}) $. For more information about the fixed-place ideals and anti fixed-place ideals, see \cite{aliabad2013fixed,aliabad2018bourbaki}.

Let $ G = \big< V(G) , E(G) \big> $ be an undirected graph.   A vertex is called a \emph{pendant vertex} if it is adjacent to just one vertex . For each pair of vertices $ u $ and $ v $ in $ V(G) $, the length of the shortest path between $ u $ and $ v $, is denoted by $ d(u,v) $, is called the \emph{distance} between $ u $ and $ v $. The \emph{eccentricity} of a vertex $ u $ of $ G $ is denoted by $ \Ec(u) $ and is defined to be maximum of $ \{ d(u,v) : u \in G \} $. The minimum of $ \{ \Ec(u) : u \in G \} $, denoted by $ \Ra(G) $, is called the \emph{radius} of $G$. We say $ G $ is \emph{triangulated}  if each vertex of $ G $ is vertex of some triangle. Two vertices $ u $ and $ v $ are called \emph{orthogonal}, if $ u $ and $ v $ are adjacent and there are not any vertex which is adjacent to the both vertices $ u $ and $ v $.   A \emph{graph homomorphism} $ \varphi $  from a graph $ G = \big< V(G), E(G) \big>$ to a graph $ H = \big< V(H), E(H) \big>$, is a map from $V(G)$ to $V(H)$ that $ \{ u, v \} \in E(G)$ implies $ \{ f(u) , f(v) \} \in E(H)$, for all pairs of vertices $u, v \in V(G)$.  A \emph{retraction} is a homomorphism $ \varphi $ from a graph $ G $ to a subgraph $ H $ of $ G $ such that $ \varphi (v) = v $, for each vertex $ v  \in V(H)‌$. In this case the subgraph $‌ H $ is called a \emph{retract} of $ G$. A subset $ D $ of vertex of a graph is called a \emph{dominating set} if every vertex of graph is either in $ D $ or adjacent to some vertex of $ D $. Also, a \emph{total dominating set} of a graph is a family $ S $ of vertex of graph such that every vertex is adjacent to some vertex of $ S $. The \emph{dominating number} and \emph{total dominating number} of a graph is the minimum cardinality of dominating set and total dominating set of graph, respectively. We denote the dominating number and total dominating number of a graph $ G $ by $ \dt(G) $ and $ \dt_t(G) $, respectively. For every $ u, v \in V(G) $, we denote the length of the shortest cycle containing $ u $ and $ v $ by $ gi(u,v) $.  

Suppose $ I $ and $ a $ are an ideal  and element of $ R $, respectively. If $ \An(I) \neq \{ 0 \} $, then $ I $  is called \emph{annihilating-ideal} and if $ \An(a) \neq \{ 0 \} $, then $ a $ is called a \emph{zero-divisor element}. Let $ \A{R} $ be the family of all non-zero annihilating-ideals and $ Z(R)^* $ be the family of all non-zero zero-divisor element of $ R $. $ \AG(R) $ is a graph with the vertices $ \A{R} $, and two distinct vertices $ I $ and $ J $ are adjacent, if $ IJ = \{ 0 \} $.  Also, $ \Gamma(R) $ is a graph with vertices $ Z(R)^* $, and  two distinct vertices $ a $ and $ b $ are adjacent, if $ ab= 0 $. $ \AG(R) $ and $ \Gamma(R) $ are called the \emph{annihilating-ideal graph} and the \emph{zero-divisor graph} of $ R $, respectively.

Thorough this paper, all $ Y \subseteq \Sp(R) $ is considered by Zariski topology; i.e.,  by assuming as a base for the closed sets of $Y$, the sets $h_Y(a)$ where $ h_Y(a) = \{ P \in Y: a \in P \} $. Hence, the closed sets of $Y$ are of the form $ h_Y(I)=\bigcap_{a\in I} h_Y(a) = \{P\in Y: I \subseteq P \} $, for some ideal $ I $ in $ R $. Also, we set $ h_Y^c(I) = Y \setminus h_Y(I) $. When $Y=\Mi(R)$ we write  $h_m$ instead of $h_Y$. A point $ P \in \Sp(R) $ is called a \emph{quasi-isolated} point, if $ P $ is an isolated point of $ \Mi(R) $. By \cite[Theorem 2.3 and Corollary 2.4]{henriksen1965space}, the space $ \Mi(R) $ is a Hausdorff space in which $ \{ h_m(a) :‌ a \in R \} $ is base of clopen sets.

In this research, $ C(X) $ denotes the ring of all real-valued continuous functions on a Tychonoff space $ X $ and we abbreviate $ \A{C(X)} $ and $ \AG(C(X)) $ by $ \AC $ and $ \AGC $, respectively. We denote the set of all isolated point of $ X $, by $ I(X) $. A space $ X $ is called almost discrete, if $ \overline{I(X)} = X $.

The reader is referred to \cite{atiyah1969introduction,sharp2000steps,stephen1970general,gillman1960rings,bondy1976m} for undefined terms and notations.

The researchers tried to define a graph illustration for some kind of mathematical aspects. For example \cite{afkhami2014graph} in the lattice literature, \cite{assari2016graphs} in the measure literature, \cite{badie2020annihilating} in topology literature and \cite{assasri2020zero} in the linear algebra. The study of translating graph properties to algebraic properties is an interesting subject for mathematicians. The introducing and studying of the concept of zero-divisor graph of a commutative is started in \cite{beck1988coloring}. In this article the author let all elements of the commutative ring be vertices of the this graph. In \cite{anderson1999zero}, it has been studied the zero-divisor graph whose vertices are the non-zero zero-divisor elements. Studying of this graph has been continued in several articles; see \cite{levy2002zero,anderson2003zero,akbari2003zero,akbari2004zero,redmond2006central,samei2007zero}. Also, First the annihilating-ideal graph has been introduced and studied in \cite{behboodi2011annihilating} and then it has been studied in several articles; see \cite{behboodi2011annihilating2,aliniaeifard2012rings,aalipour2012coloring,aalipour2014classification,nikandish2014dominating,chelvam2015connectivity,nikandish2015domination}.

In the rest of this section we give  a retract of the annihilating graph. Section 2, devoted to translating the graph properties of these graphs to Zariski topology. Also, we note an impossible assumption in \cite{samei2007zero}. In Section 3, by obtained tools in Section 2, we characterize the radius of $ \Gamma(R) $, $ \AG(R) $, $ \Gamma(X) $ and $ \AG(X) $ and show that $ \Ra(\Gamma(R))  $ and $ \Ra(\AG(R)) $ are equal and they are equal to 3, \ff the zero ideal of $ R $ is an anti fixed-place ideal, \ff $ \Mi(R) $ does not have any isolated point, \ff $ \Gamma(R) $ is triangulated, \ff $ \AG(R) $ is triangulated. In the last section, the domination number of the annihilating-ideal graph has been studied. In this section we show that $ |\B(R)| \leqslant \dt(AG(R)) $. Also, we note a mistake of \cite{nikandish2015domination} and we characterize the domination of a ring in which the zero ideal is a fixed-place ideal and domination of $ \AG(X) $ in which $ X $ is almost discrete and finally we prove that $ \dt(\AG(R)) $ is finite, \ff $ \dt_t(\AG(R)) $ is finite; \ff $ \Mi(R) $ is finite. 

For each subset $ S $ of $ R $ let $ P_{_S} $ be the intersection of all minimal prime ideals containing $ S $. An ideal $ I $ in $ R $ is said to be \emph{strongly $ z^\circ $-ideal} (or briefly \emph{$ sz^\circ $-ideal}) if $ P_{_F} \subseteq I $, for every finite subset $ F $ of $ I $. Since the intersection of every family of strong $ z^\circ $-ideals is a strong $ z^\circ $-ideal, the smallest strong $ z^\circ $-ideal containing an ideal $ I $ exists, and we denote this by $ I_{sz^\circ} $. For more details about the strong $ z^\circ $-ideals, see \cite{mason1989prime,aliabad2011sz,badie2019extension}. 

\begin{Lem}
	Let $ I $ and $ J $ be ideals of $ R $. $ I $ is adjacent to $ J $, \ff $ I_\sz $ is adjacent to $ J_\sz $.
	\label{I_sz0 is adjacent to J_sz0}
\end{Lem}
\begin{proof}
	$ \Rightarrow $). Suppose that $ a \in I_\sz $ and $ b \in J_\sz $, then , by \cite[Proposition 7.5]{badie2019extension},  finite subsets $ F $ of $ I $ and $ G $ of $ J $ exist such that $ h_m(G) \subseteq h_m(a) $ and $ h_m(H) \subseteq h_m(b) $. Since $ I $ is adjacent to $ J $, $ IJ = \{ 0 \} $,  so $ GH = \{ 0 \} $, this implies that $  Min(R) = h_m(GH) = h_m(G) \cup h_m(H) \subseteq h_m(a) \cup h_m(b) = h_m(ab) $, thus $ h_m(ab) = \Mi(R) $, hence $ ab \in kh_m(ab) = \{ 0 \} $, and therefore $ ab = 0 $. This shows that $ I_\sz J_\sz = \{ 0 \} $ and therefore $ I_\sz $ is adjacent to $ J_\sz $.
	
	$ \Leftarrow $). It is clear.
\end{proof}

\begin{Pro}
	The family of all $ \sz $-ideals of $ \A{R} $ is a retract of $ \AG(R) $.
\end{Pro}
\begin{proof}
	Suppose that $ I \in \A{R} $, so $ J \in \A{R} $ exists such that $ IJ = \{ 0 \} $. By Lemma \ref{I_sz0 is adjacent to J_sz0}, $ I_\sz $ is adjacent to $ J_\sz $. Since $ 0 \neq I \subseteq I_\sz \subseteq \An(J_\sz) \subseteq \An(J) \neq X   $, $ I_\sz \in \A{R} $. This shows that the map $ \varphi $ from $ \A{R} $ to the family of all $ \sz $-ideals of $ \A{R} $, defined by $ \varphi(I) = I_\sz $ is a retraction and therefore the family of all $ \sz $-ideals of $ \A{R} $ is a retract of $ \AG(R) $.
\end{proof}

\section{Zariski topology}

In this section we give Zariski topological characterization of elements of $ \Gamma(R) $ and $ \AG(R) $, then we characterize the adjacency, distance, orthogonality, eccentricity and triangulation of vertices of these graphs. Also, it has been shown that $ \Ra\Gamma(R),\Ra\AG(R) > 1 $.   

\begin{Pro}
	Let $ Y \subseteq \Sp(R) $ and $ \bigcap Y = \{ 0 \} $. If $ a $ is an element and $ I $ is an ideal of $ R $, then
	\begin{itemize}
		\item[(a)] $ a = 0 $, \ff $ h_Y(a) = Y $.
		\item[(b)] $ \An(a) \neq 0 $, \ff $ \overline{h_Y^c(a)} \neq Y $.
		\item[(c)] $ I = \{ 0 \} $, \ff $ h_Y(I) = Y $.
		\item[(d)] $ I $  is an annihilating-ideal, \ff $ \overline{h_Y^c(I)} \neq Y $. 
	\end{itemize}
\label{elements}
\end{Pro}
\begin{proof}
	(a) and (c). Since $ \bigcap Y = \{ 0 \} $, They are clear.
	
	(b). Since $ \An(a) = k h_Y^c(a) $, $ \An(a) \neq \{ 0 \} $ \ff $ kh_Y^c (a) \neq \{ 0 \} $; and it is equivalent to say that $ hkh_Y^c(a) \neq Y $, because $ \bigcap Y = \{ 0 \} $, and therefore it is equivalent to $ \overline{h_Y^c(I)} \neq Y $.
	
	(d). The proof is analogously similar to the proof part (b).
\end{proof}

\begin{Lem}
	Let $ Y \subseteq \Sp(R) $ and $ \bigcap Y = \{ 0 \} $.
	\begin{itemize}
		\item[(a)] For each $ a,b \in Z(R)^* $, $ a $ is adjacent to $ b $, \ff $ h_Y^c(a) \cap h_Y^c(b) = \emptyset $.
		\item[(b)] For each $ I,J \in \A{R} $, $ I $ is adjacent to $ J $, \ff $ h_Y^c(I) \cap h_Y^c(J) = \emptyset $.
	\end{itemize}
\label{adjacent}
\end{Lem}
\begin{proof}
	It is evident.
\end{proof}

In \cite[Proposition 2.2]{samei2007zero}, the concept of distance in $ \Gamma(R) $ has characterized by Zariski topology on $ \Sp(R) $. In the following proposition we generalize this characterization by every reduced family of prime ideals and also we characterize the concept of distance in $ \AG $.

\begin{Pro}
	Let $ I, J \in \A{R} $, $ a,b \in Z(R)^* $, $ Y \subseteq \Sp(R) $ and $ \bigcap Y = \{ 0 \} $. Then
	\begin{itemize}
		\item[(a)] $ d(a,b) = 1 $, \ff $ h_Y^c(a) \cap h_Y^c(b) = \emptyset $.
		\item[(b)] $ d(a,b) = 2 $, \ff $ h_Y^c(a) \cap h_Y(b) \neq \emptyset $ and $ h_Y^c(a) \cup h_Y^c(b) $ is not dense in $ Y $.
		\item[(c)] $ d(a,b) = 3 $, \ff $ h_Y^c(a) \cap h_Y^c(b) \neq \emptyset $ and $ h_Y^c(a) \cup h_Y^c(b) $ is dense in $ Y $.
		\item[(d)] $ d(I,J) = 1 $, \ff $ h_Y^c(I) \cap h_Y^c(J) = \emptyset $.
		\item[(e)] $ d(I,J) = 2 $, \ff $ h_Y^c(I) \cap h_Y^c(J) \neq \emptyset $ and $ h_Y^c(I) \cup h_Y^c(J) $ is not dense in $ Y $.
		\item[(f)] $ d(I,J) = 3 $, \ff $ h_Y^c(I) \cap h_Y^c(J) \neq \emptyset $ and $ h_Y^c(I) \cup h_Y^c(J) $ is dense in $ Y $.
	\end{itemize}
\label{distance}
\end{Pro}
\begin{proof}
	(a) and (d). They are clear, by Lemma \ref{adjacent}.
	
	(b $ \Rightarrow $). By Lemma \ref{adjacent}, $ h_Y^c(a) \cap h_Y^c(b) \neq \emptyset $. By the assumption, there is an ideal $ c \in Z(R)^* $, such that $ c $ is adjacent to the both vertices $ a $ and $ b $. Now Lemma \ref{adjacent}, implies that 
	\[ h_Y^c(a) \cap h_Y^c(c) = h_Y^c(a) \cap h_Y^c(c) = \emptyset \quad \Rightarrow \quad h_Y^c(a) \cup h_Y^c(b) \subseteq h_Y(c) \quad (*) \]
	Since $ c \neq 0 $, by Proposition \ref{elements}, $ h_Y(c) \neq Y  $, and since $ h_Y(c) $ is closed, $(*)$ follows that $ h_Y^c(a) \cup h_Y^c(b) $ is not dense in $ Y $.

	(b $ \Leftarrow $). By part (a), $ d(a,b) > 1 $. Since $ \{ h_Y^c(c) : c \in R \} $ is a base for Zariski topology, by the assumption, there is some $ c \in R $ such that $ h_Y^c(a) \cup h_Y^c(b) \subseteq h_Y(c) \subset Y $, so $ h_Y^c(a) \cap h_Y^c(c) = h_Y^c(a) \cap h_Y^c(c) = \emptyset  $, $ Y \neq h_Y(a) $ and $ \overline{h_Y^c(c)} \neq Y $, thus $ c \in Z(R)^* $ and $ c $ is adjacent to the both vertices $ a $ and $ b $, hence $ d(a,b) = 2 $.

	(c). It deduces from parts (a), (b) and \cite[Theorem 2.2]{anderson1999zero}.
	
	(e). By this fact that $ \{ h_Y^c(K) : K \text{ is an ideal of } R \} $ is a base for Zariski topology, it is similar to part (b)
	
	(f). It concludes from parts (d), (e) and \cite[Theorem 7.1]{behboodi2011annihilating}.
\end{proof}

\begin{Thm}
	Let $ I, J \in \A{R} $, $ a , b \in Z(R)^* $, $ Y \subseteq \Sp(R) $ and $ \bigcap Y = \{ 0 \} $. Then 
	\begin{itemize}
		\item[(a)] Two vertices $ I $ and $ J $ are orthogonal, \ff $ h_Y^c(I) \cap h_Y^c(J) = \emptyset $ and $ \overline{h_Y^c(I) \cup h_Y^c(J)} = Y $.
		\item[(b)] Two vertices $ a $ and $ b $ are orthogonal, \ff $ h_Y^c(a) \cap h_Y^c(b) = \emptyset $ and $ \overline{h_Y^c(a) \cup h_Y^c(b)} = Y $.
	\end{itemize}
\label{orthogonal}
\end{Thm}
\begin{proof}
	(a $ \Rightarrow $). By the assumption and Lemma \ref{adjacent}, $ I $ is adjacent to $ J $, so $ h_Y^c(I) \cap h_Y^c(J) = \emptyset $. If $ \overline{h_Y^c(I) \cup h_Y^c(J)} \neq Y $, since $ \{ h_Y^c(K) : K \text{ is an ideal of } R \} $ is a base for Zariski topology, it follows that there is some ideal $ K $ of $ R $ such that $ h_Y^c(K) \cap \left[ h_Y^c(I) \cup h_Y^c(J) \right] = \emptyset $, so $ h_Y^c(K) \cap h_Y^c(I)  = h_Y^c(K) \cap h_Y^c(J) = \emptyset $, $ h_Y^c(K) \neq Y $ and $ \overline{h_Y^c(K)} \neq Y $, thus $ K \in \A{R} $, by Proposition \ref{elements}, and $ K $ is adjacent to the both vertices $ I $ and $ J $, by Lemma \ref{adjacent}, which contradicts the assumption, hence $ \overline{h_Y^c(I) \cup h_Y^c(J)} = Y $.
	
	(a $ \Leftarrow $). By the assumption and Lemma \ref{adjacent}, $ h_Y^c(I) \cap h_Y^c(J) = \emptyset $. On contrary, suppose that there is an $ K \in \A{R} $, such that $ K $ is adjacent to the both vertices $ I $ and $ J $, then $ h_Y^c(K) \cap \left[ h_Y^c(I) \cup h_Y^c(J) \right] = \left[ h_Y^c(K) \cap h_Y^c(I) \right] \cup \left[ h_Y^c(K) \cap h_Y^c(J) \right] = \emptyset $, by Lemma \ref{adjacent}. Since $ K \in \A{R} $, by Proposition \ref{elements}, $ h_Y^c(K) \neq \emptyset $, and therefore $ \overline{h_Y^c(I) \cup h_Y^c(J)} \neq Y $, which contradicts the assumption.
	
	(b). By this fact $ \{ h_Y^c(c) : c \in R \} $ is a base for Zariski topology, it is similar to part (a).
\end{proof}

Suppose that $ \bigcap Y = \{ 0 \} $. Since for every $ I \in \A{R} $, $ I $ and $ \An(I) $ are orthogonal, the above theorem implies that $ \overline{ h^c_Y(I) \cap h^c_Y(\An(I)) } = Y  $. Similarly, for every $ a \in Z(R)^* $ and $ b \in \An(a) $, we have $ \overline{h_Y^c(a) \cup h_Y^c(b)} = Y $.

\begin{Thm}
	Suppose that $ I \in \A{R} $, $ a \in Z(R)^* $, $ Y \subseteq \Mi(R) $ and $ \bigcap Y = \{ 0 \} $. Then
	\begin{itemize}
		\item[(a)] For every $ I \in \A{R} $, $ \Ec(I)>1 $.
		\item[(b)] $ \Ec(I) = 2 $, \ff $ h_Y^c(I) $ is singleton.
		\item[(c)] $ \Ec(I) = 3 $, \ff $ h_Y^c(I) $ is not singleton.
		\item[(d)] For every $ a\in Z(R)^* $, $ \Ec(a) > 1 $.
		\item[(e)] $ \Ec(a) = 2 $, \ff $ h_Y^c(a) $ is singleton.
		\item[(f)] $ \Ec(a) = 3 $, \ff $ h_Y^c (a) $ is not singleton.
	\end{itemize}
\label{ecc}
\end{Thm}
\begin{proof}
	Since $ R $ is not an integral domain and $ \bigcap Y  = \{ 0 \} $, it follows that $ | Y | \geqslant 2 $.
		
	(c $ \Rightarrow $). By the assumption there is some $ J \in \A{R} $ such that $ d(I,J) = 3 $. Lemma \ref{distance}, implies that $ h_Y^c(I) \cap h_Y^c(J) = \emptyset $ and $ \overline{h_Y^c(I) \cup h_Y^c(J)} = Y $. On contrary, suppose that $ h_Y^c(I) $ is singleton, then $ h_Y^c(I) \subseteq h_Y^c(J) $ and therefore $ \overline{h_Y^c(J)} = \overline{ h^c_Y(I) \cup h^c_Y(J) } = Y $, so $ J \notin \A{R} $, by Lemma \ref{adjacent}, which is a contradiction.
	
	(c $ \Leftarrow $). By the assumption, there are distinct prime ideals $ P $ and $ Q $ in $ h_Y^c(I) $. Since $ Y \subseteq \Mi(R) $ is Hausdorff and $ \{ h_Y^c(K) : K \text{ is an ideal of } R \} $ is a base for $ Y $, there are ideals $ J $ and $ K $ such that $ h_Y^c(J) , h_Y^c(K) \subseteq h_Y^c(I) $, $ P \in h_Y^c(J) $, $ Q \in h_Y^c(K) $ and $ h_Y^c(J) \cap h_Y^c(K) = \emptyset $. Thus 
		\begin{align*}
		h_Y^c(J + \An(I)) \cap h_Y^c(K) & = \left[ h_Y^c(J) \cup h_Y^c(\An(I)) \right] \cap h_Y^c(K) \\
								 		& \subseteq \left[ h_Y^c(J) \cap h_Y^c(K) \right] \cup \left[ h_Y^c(\An(I)) \cap h_Y^c(I) \right] = \emptyset.
		\end{align*}
	Hence $ h_Y^c(J + \An(I)) \neq Y $ and $ \overline{h_Y^c(J + \An(I))} \neq Y $, so $ J + \An(I) \in \A{R} $. Since 
	\[ h_Y^c(I) \cap h_Y^c(J + \An(I)) \supseteq h_Y^c(I) \cap h_Y^c(J) = h_Y^c(I) \neq \emptyset \]
	and 
	\[ \overline{ h^c_Y(I) \cap h^c_Y(J + \An(I)) } \supseteq \overline{ h^c_Y(I) \cap h^c_Y(\An(I)) } = Y, \]
	by Proposition \ref{distance}, $ d(I,J + \An(I)) = 3 $ and therefore $ \Ec(I) = 3 $, by \cite[Theorem 7.1]{behboodi2011annihilating}.
	
	(a). Suppose that there is some $ I \in \A{R} $ such that $ \Ec(I) = 1 $. By part (c), $ h_Y^c(I) $ is singleton, so there is some $ P \in Y $, such that $ h_Y^c(I) = \{ P \} $, thus $ \An(I) = P $, hence $ \{ 0 \} \neq I \subseteq \An(P) $. Since $ \Ec(I) = 1 $, $ I $ is adjacent to $ \An(P) $, consequently $ I \An(P) = \{ 0 \} $, this implies that for every $ a \in I $, $ a^2 \in I \An(P) = \{ 0 \} $, and therefore $ a^2 = 0 $. Since $ R $ is reduced, $ a = 0 $, and consequently $ I = \{ 0 \} $, which is a contradiction.
	
	(b). By parts (a), (c) and \cite[Theorem 7.1]{behboodi2011annihilating}, it is clear.
	
	The proof of (d), (e) and (f) are similar to parts (a), (b) and (c), respectively.
\end{proof}

The following corollary is an immediate consequence of the above theorem. 

\begin{Cor}
	$ \Ra \Gamma(R) > 1 $ and $ \Ra \AG(R) > 1 $.
\label{radius}
\end{Cor}

\begin{Pro}
	Let $ a \in Z(R)^* $, $ I \in \A{R} $, $ Y \subseteq \Mi(R) $ and $ \bigcap Y = \{ 0 \} $. Then
	\begin{itemize}
		\item[(a)] $ a $ is a vertex of a triangle, \ff $ h_Y(a) $ is not singleton.
		\item[(b)] $ I $ is a vertex of a triangle, \ff $ h_Y(I)^\circ $ is not singleton.
	\end{itemize}
\label{triangle}
\end{Pro}
\begin{proof}
	(a $ \Rightarrow $). By the assumption, there are vertices $ b , c \in \A{R} $ such that $ a $, $ b $ and $ c $ are pairwise vertices which are adjacent together. Thus $ h_Y^c(a) $, $ h_Y^c(b) $ and $ h_Y^c(c) $ are pairwise disjoint nonempty sets, by Theorem \ref{adjacent} and Proposition \ref{elements}, hence $ h_Y^c(b) \cup h_Y^c(c) \subseteq h_Y(a) $ and $ | h_Y^c(b) \cup h_Y^c(c) | \geqslant 2 $, since $ h_Y^c(b) \cup h_Y^c(c) $ is open, it follows that $ h_Y(a) $ is not singleton.
	
	(a $ \Leftarrow $). Suppose that $ P $ and $ Q $ are distinct elements of $ h_Y(a) $. Since $ Y \subseteq \Mi(R) $ is Hausdorff, $ h_Y(a) $ is open and $ \{ h_Y^c(x) : x \in R \} $ is a base for $ Y $, there are $ b , c \in R $ such that $ P \in h_Y^c(b) \subseteq h_Y(a) $, $ Q \in h_Y^c(c) \subseteq h_Y(a) $ and $ h_Y^c(b) \cap h_Y^c(c) = \emptyset $, so $ h_Y^c(a) $, $ h_Y^c(b) $ and $ h_Y^c(c) $ are pairwise disjoint nonempty sets which are not dense in $ Y $. Now Proposition \ref{elements}, implies that $ b,c \in \A{R} $ and Theorem \ref{adjacent}, concludes that $ a $, $ b $ and $ c $ are pairwise vertices which are adjacent together, hence $ a $ is a vertex of a triangle.
	
	(b). It is similar to part (a).
\end{proof}

\begin{Pro}
	Suppose that $ a,b \in Z(R)^* $ are not pendant vertices, $ Y \subseteq \Mi(R) $ and $ \bigcap Y = \{ 0 \} $. Then
	\begin{itemize}
		\item[(a)] $ h_Y^c(a) \cap h_Y^c(b) = \emptyset $ and $ \overline{h_Y^c(a) \cup h_Y^c(b)} \neq Y $, \ff $ \gi(a,b) = 3 $.
		\item[(b)] If $ 2 \notin Z(R) $, $ h_Y^c(a) \cap h_Y^c(b) = \emptyset $ and $ \overline{h_Y^c(a) \cup h_Y^c(b)} = Y $, then $ \gi(a,b) = 4 $.
		\item[(c)] Suppose that $ h_Y^c(a) \cap h_Y^c(b) \neq \emptyset $. Then $ \overline{h_Y^c(a) \cup h_Y^c(b)} \neq Y $, \ff $ \gi(a,b) = 4 $.
		\item[(d)] Suppose that $ 2 \notin Z(R) $ and $ h_Y^c(a) \cap h_Y^c(b) \neq \emptyset $. Then $ \overline{h_Y^c(a) \cup h_Y^c(b)} = Y $, \ff $ \gi(a,b) = 6 $.
	\end{itemize}
\end{Pro}
\begin{proof}
	By Proposition \ref{elements} and Lemma \ref{adjacent}, it has a similar proof to \cite[Theorem 3.4]{samei2007zero}.
\end{proof}

\begin{Thm}
	Suppose that $ I , J \in \A{R} $ and they are not pendant vertices. The following statements hold.
	\begin{itemize}
		\item[(a)] $ h_Y^c(I) \cap h_Y^c(J) = \emptyset $ and $ \overline{h_Y^c(I) \cup h_Y^c(J)} \neq Y $, \ff $ \gi(I,J) = 3 $.
		\item[(b)] If $ h_Y^c(I) \cap h_Y^c(J) = \emptyset $ and $ \overline{h_Y^c(I) \cup h_Y^c(J)} = Y $, then $ \gi(I,J) = 4 $.
		\item[(c)] If $ h_Y^c(I) \cap h_Y^c(J) \neq \emptyset $ and $ \overline{h_Y^c(I)} = \overline{h_Y^c(J)} $, then $ \gi(I,J) = 4 $.
		\item[(d)] If $ h_Y^c(I) \cap h_Y^c(J) \neq \emptyset $ and $ \overline{h_Y^c(I)} \neq \overline{h_Y^c(J)} $ and $ Y \setminus \overline{h_Y^c(I) \cup h_Y^c(J)} $ is not singleton, then $\gi(I,J) = 4$.
		\item[(e)] If $ h_Y^c(I) \cap h_Y^c(J) \neq \emptyset $, $ \overline{h_Y^c(I)} \neq \overline{h_Y^c(J)} $ and $ Y \setminus \overline{h_Y^c(I) \cup h_Y^c(J)} $ is singleton, then $4 \leqslant \gi(I,J) \leqslant 5 $.
		\item[(f)] If $ \gi(I,J) = 5 $, then $ h_Y^c(I) \cap h_Y^c(J) \neq \emptyset $, $ \overline{h_Y^c(I)} \neq \overline{h_Y^c(J)} $ and $ Y \setminus \overline{h_Y^c(I) \cup h_Y^c(J)} $ is singleton.
	\end{itemize}
\end{Thm}
\begin{proof}
	(a $\Rightarrow$). By Lemma \ref{adjacent}, $ I $ is adjacent to $ J $ and by Theorem \ref{orthogonal}, $ I $ and $ J $ are not orthogonal. Thus $ \gi(I,J) = 3 $.
	
	(a $\Leftarrow$). Then $ I $ is adjacent to $ J $ and the vertices $ I $ and $ J $ are not orthogonal, so by Lemma \ref{adjacent}, we have $ h_Y^c(I) \cap h_Y^c(J) = \emptyset $ and by Proposition \ref{orthogonal}, $ \overline{h_Y^c(I) \cup h_Y^c(J)} \neq Y $.
	
	(b). By the assumption $ I J = \{ 0 \} $, and we can see easily that $ h_Y(I)^\circ \cap h_Y(J)^\circ = \emptyset $, we know that $ h_Y^c(\An(I)) \subseteq h_Y(I)^\circ $ and $ h_Y^c(\An(J)) \subseteq h_Y(J)^\circ $, so $ h_Y^c(\An(I)) \cap h_Y^c(\An(J)) = \emptyset $. Now Lemma \ref{adjacent}, concludes that $ \An(I) \An(J) = \{ 0 \} $. Since $ I $ and $ J $ are not pendant vertices, there are $ I_1 , J_1 \in \AC $ such that $ I $ is adjacent to $ I_1 \neq J $ and $ J $ is adjacent to $ J_1 \neq I $, so $ I I_1 = J J_1 = \{ 0 \} $, thus $ I_1 \subseteq \An(I) $ and $ J_1 \subseteq \An(J) $, hence $ I_1 J_1 \subseteq \An(I) \An(J) = \{ 0 \} $ and therefore $ I_1 J_1 = \{ 0 \} $. Consequently, $ I $ is adjacent to $ J $, $ J $ is adjacent to $ J_1 $, $ J_1 $ is adjacent to $ I_1 $ and $ I_1 $ is adjacent to $ I $, they imply that $ \gi(I,J) = 4 $.
		
	(c). We can conclude from the assumption and part (a) that $ \gi(I,J) \geqslant 4  $. Clearly $ \An(I) , \An(J) \in \A{R} $. Since $ \overline{h_Y^c(I)} = \overline{h_Y^c(J)} $, it follows that $ h_Y^c(I) \cap h_Y^c(\An(J)) \subseteq \overline{h_Y^c(I)} \cap h_Y(J)^\circ = \overline{h_Y^c(I)} \cap \overline{h_Y^c(J)}^{^c} = \overline{h_Y^c(I)} \cap \overline{h_Y^c(I)}^{^c} = \emptyset $,
	so, by Lemma \ref{adjacent}, $ I \An(J) = \{ 0 \} $.  Similarly, we can show that $ J \An(I) = \{ 0 \} $. If $ \An(I) \neq \An(J) $, then $ I $ is adjacent to $ \An(I) $, $ \An(I) $ is adjacent to $ J $, $ J $ is adjacent $ \An(J) $ and $ \An(J) $ is adjacent to $ I $ and therefore $ \gi(I,J) = 4 $. Now we suppose that $ \An(I) = \An(J) $. Since $ I $ is adjacent to $ \An(I) $ and $ I $ is not a pendant vertex, it follows there is some vertex $  I_1 \in \AC $ distinct from $ \An(I) $ such that $ I $ is adjacent to $ I_1 $, then $ I_1 I = \{ 0 \} $, so $ I_1 \subseteq \An(I) = \An(J) $ and therefore $ I_1 J = \{ 0 \} $. Consequently, $ I $ is adjacent to $ \An(I) $, $ \An(J) $ is adjacent to $ J $, $ J $ is adjacent to $ I_1 $ and $ I_1 $ is adjacent to $ I $ and thus $ \gi(I,J)= 4 $.
	
	(d). We can conclude from the assumption and part (a) that $ \gi(I,J) \geqslant 4  $. Since $ \{ h_Y^c(K) : K \text{ is an ideal of } R \} $ is a base for $ Y $, $ Y $ is Hausdorff and $ Y \setminus\overline{ h^c_Y(I) \cap h^c_Y(J) } $ is not singleton, it follows that there are two distinct ideals $ K_1 $ and $ K_2 $ such that $ h_Y^c(K_1) \cap ( h_Y^c(I) \cup h_Y^c(J) ) = h_Y^c(K_2) \cap ( h_Y^c(I) \cup h_Y^c(J) ) = \emptyset $. Hence $ h_Y^c(I) \cap h_Y^c(K_1) = h_Y^c(K_1) \cap h_Y^c(J) = h_Y^c(J) \cap h_Y^c(K_2) = h_Y^c(K_2) \cap h_Y^c(I) =\emptyset $. Then, by Theorem \ref{elements}, $ K_1 , K_2 \in \A{R} $, and by Lemma \ref{adjacent}, $ I $ is adjacent to $ K_1 $, $ K_1 $ is adjacent to $ J $, $ J $ is adjacent to $ K_2 $ and $ K_2 $ is adjacent to $ I $. Consequently, $ \gi(I,J) = 4 $.
	
	(e). By part (a), $ \gi(I,J) \geqslant 4 $. Since $ Y \setminus \overline{h_Y^c(I) \cup h_Y^c(J)} \neq Y $ and $ \{ h_Y(K) : K \text{ is an ideal of } R \} $ is a base for $ Y $, it follows that there is some  ideal $ K_1 $ of $ R $ such that $ h_Y^c(K_1) \cap [ h_Y^c(I) \cup h_Y^c(J) ] = \emptyset $, so $ h_Y^c(K_1) \cap h_Y^c(I) = h_Y^c(K_1) \cap h_Y^c(J) = \emptyset $. By Theorem \ref{elements}, $ K_1 \in \A{R} $ and Lemma \ref{adjacent}, concludes that $ K_1 $ is adjacent to the both vertices $ I $ and $ J $. If there is an $ K_2 \in \A{R} $ distinct from $ K_1 $ such that $ h_Y^c(K_1) = h_Y^c(K_2) $, then $ K_2 $ also is adjacent to the both vertices $ I $ and $ J $. Thus $ \gi(I,J) = 4 $. Now suppose that $ h_Y^c(K) = h_Y^c(K_1) $ implies that $ K = K_1 $. If $ h_Y^c(I) \subseteq \overline{h_Y^c(J)} $, then $ \overline{h_Y^c(I)} \subseteq \overline{h_Y^c(J)} $, so $ Y \setminus \overline{h_Y^c(I) \cup h_Y^c(J)} = Y \setminus \overline{h_Y^c(J)} $ and therefore, by the assumption, $ Y \setminus \overline{h_Y^c(J)} $ is singleton. Since $ J $ is not a pendant vertex, there is some vertex $ K_2 $ such that $ K_2 $ is adjacent to $ J $, thus, by Lemma \ref{adjacent}, $ h_Y^c(K_2) \cap h_Y^c(J) = \emptyset $, so $ h_Y^c(K_2) \cap \overline{h_Y^c(J)} = \emptyset $, thus $ h_Y^c(K_2) \subseteq Y \setminus \overline{h_Y^c(J)} $. By Theorem \ref{elements}, $ h_Y^c(K_2) \neq \emptyset $ and therefore $ h_Y^c(K_2) = Y \setminus \overline{h_Y(J)} $. Similarly, we can show that $ h_Y^c(K_1) = Y \setminus \overline{h_Y(J)} $, hence $ h_Y^c(K_1) = h_Y^c(K_2) $, which is a contradiction. Hence $ h_Y^c(I) \not\subseteq \overline{h_Y^c(J)} $. Similarly one can show $ h_Y^c(J) \not\subseteq \overline{h_Y^c(I)} $, thus $ h_Y^c(I) \setminus \overline{h_Y^c(J)} $ and $ h_Y^c(J) \setminus \overline{h_Y^c(I)} $ are disjoint nonempty open sets. Since $ \{ h_Y(K) : K \text{ is an ideal of } R \} $ is a base for $ Y $, there are distinct ideals $ K_2 $ and $ K_3 $ such that $ h_Y^c(K_2) \subseteq h_Y^c(I) \setminus \overline{h_Y^c(J)} $ and $ h_Y^c(K_3) \subseteq h_Y^c(J) \setminus \overline{h_Y^c(I)} $. Consequently, $ h_Y^c(J) \cap h_Y^c(K_2) = h_Y^c(K_2) \cap h_Y^c(K_3) = h_Y^c(K_3) \cap h_Y^c(I) = \emptyset $. By Theorem \ref{elements}, we have $ K_2 , K_3 \in \A{R} $ and Lemma \ref{adjacent} concludes that $ I $ is adjacent to $ K_1 $, $ K_1 $ is adjacent to $ J $, $ J $ is adjacent to $ K_2 $, $ K_2 $ is adjacent to $ K_3 $ and $ K_3 $ is adjacent to $ I $, and therefore $ \gi(I,J) \leqslant 5 $.
	
	(f). It is clear, by parts (a)-(e).
\end{proof}

Suppose that $ R = \Z \times \Z \times \Z \times \Z $, $ I = \{ 0 \} \times \Z \times \Z \times \{ 0 \} $, $ J = \Z \times \{ 0 \} \times \Z \times \{ 0 \} $, $ R' = \R \times \R \times \R \times \R $, $ I' = \{ 0 \} \times \R \times \R \times \{ 0 \} $, $ J = \R \times \{ 0 \} \times \R \times \{ 0 \} $. Then the both pair vertices $ I, J \in \A{R} $ and $ I',J' \in \A{R'} $ satisfy in the conditions of part (e) of the above theorem but it is seen readily that $ \gi(I,J) = 4 $ and $ \gi(I',J')=5 $. 

Now we can conclude the following corollary from the above theorem and \cite[Corollary 4.2]{aliabad2013fixed}.

\begin{Cor}
	If for some $ I , J \in \A{R} $, we have $ \gi(I,J) = 5 $, then the following equivalent conditions hold
	\begin{itemize}
		\item[(a)] $ \Mi(R) $ has an isolated point.
		\item[(b)] $ \mathcal{B}(R) \neq \emptyset $.
	\end{itemize}
\end{Cor}

\section{Radius and Triangulation}

This section is has been devoted to study of the radius and the triangulation of $ \Gamma(R) $ and $ \AG(R) $. We show that the concept of the anti fixed-place ideal plays the main role in this studying.

\begin{Thm}
	The following statement are equivalent.
	\begin{itemize}
		\item[(a)] $ \Ra \Gamma(R) = 3 $.
		\item[(b)] $ \Ra \AG(R) = 3 $.
		\item[(c)] The zero ideal of $ R $ is an anti fixed-place ideal.
		\item[(d)] The $ \Mi(R) $ does not have any isolated point.
	\end{itemize}
\label{radius = 3}
\end{Thm}
\begin{proof}
	(a) $ \Rightarrow $ (b). Suppose that $ \Ra \AG(R) \neq 3 $, then, by Corollary \ref{radius} and \cite{behboodi2011annihilating}, there is some $ I \in \A{R} $ such that $ \Ec(I) = 2 $, hence, Theorem \ref{ecc}, there is some $ P \in \Mi(R) $ such that $ h_m^c(I) = \{ P \} $, thus $ \An(I) = P $. Set $ 0 \neq a \in I  $, then  $ \emptyset \neq h_m^c(a) \subseteq h_m^c(I) = \{ P \} $, so $ h_m^c(a) = \{ P \} $ and therefore $ \Ec(a) = 2 $, by Theorem \ref{ecc}. Consequently, $ \Ra \Gamma(R) \neq 3 $.
	
	(b) $ \Rightarrow $ (c). Suppose the zero ideal of $ R $ is not an anti fixed-place ideal, then there is an affiliated prime ideal $ P $, hence $ a \in Z(R)^* $ exists such that $ \An(a) = P $, this implies that $ \Ge{a} \in \A{R} $ and $ h_m^c(\Ge{a}) = h_m^c(a) = \{ P \} $ and therefore $ \Ra \AG(R) \neq 3 $, by Theorem \ref{ecc}.
	
	(c) $ \Rightarrow $ (a). Suppose that $ \Ra \Gamma(R) \neq 3 $, then, Corollary \ref{radius} and \cite{behboodi2011annihilating}, there is some $ a \in Z(R)^* $ such that $ \Ec(a) = 2 $, hence, by Theorem \ref{radius}, there is some $ P \in \Mi(R) $ such that $ h_m^c(a) = \{ P \} $, thus $ \An(a) = P $, hence $ P $ is affiliated prime ideal, so $ P \in \mathcal{B}(R) \neq \emptyset $ and therefore the zero ideal of $ R $ is not an anti fixed-place ideal.
	
	(c) $ \Leftrightarrow $ (d). It implies from \cite[Corollary 4.3]{aliabad2013fixed}.
\end{proof}

The following corollary is an immediate consequence of the above theorem and Corollary \ref{radius}.

\begin{Cor}
	The following statement are equivalent.
	\begin{itemize}
		\item[(a)] $ \Ra \Gamma(R) = 2 $.
		\item[(b)] $ \Ra \AG(R) = 2 $.
		\item[(c)] The zero ideal of $ R $ is not an anti fixed-place ideal.
		\item[(d)] The $ \Mi(R) $ has an isolated point.
	\end{itemize}
\label{radius = 2}
\end{Cor}

Now we can conclude the following corollary from the above theorem and corollary.

\begin{Cor}
	$ \Ra \Gamma(R) = \Ra \AG(R) $.
\end{Cor}

\begin{Cor}
	Suppose that $ X $ is a Tychonoff topological space. Then
	\[ \Ra \Gamma (X) = \Ra \AG(X) = \begin{cases}
		 2 & \text{ If } X \text{ has an isolated point. } \\
		 3 & \text{ If } X \text{ does not have any isolated point. } 
	\end{cases} \]
\end{Cor}
\begin{proof}
	It conclude from \cite[Corollary 5.4]{aliabad2013fixed}, Theorem \ref{radius = 3} and Corollary \ref{radius = 2}.
\end{proof}

\begin{Thm}
	The following statements are equivalent.
	\begin{itemize}
		\item[(a)] The zero ideal of $ R $ is an anti fixed-place ideal.
		\item[(b)] $ \Gamma(R) $ is triangulated.
		\item[(c)] $ \Mi(R) $ does not have any isolated point.
	\end{itemize}
\label{Gamma R is triangulated}
\end{Thm}
\begin{proof}
	(a) $ \Rightarrow $ (b). Suppose that $ \Gamma(R) $ is not triangulated, then $ a \in Z(R)^* $ exists such that $ a $ is not a vertex of any triangle, so by Proposition \ref{triangle}, $ h_m(a) $ is singleton, hence there is a $ P \in \Mi(R) $ such that $ h_m(a) = \{ P \} $. Since $ h_m(a) $ is open and $ \{ h_m^c(x) : x \in R \} $ is base for $ Y $, there is some $ b \in R $ such that $ P \in h_m^c(b) \subseteq h_m(a) = \{ P \} $, thus $ h_m^c(b) = \{ P \} $ and therefore $ \An(b) = P $. It shows that $ P $ is affiliated prime ideal, hence $ P \in \mathcal{B}(R) \neq \emptyset $ and consequently the zero ideal is not an anti fixed-place ideal.
	
	(b) $ \Rightarrow $ (c). By \cite[Theorem 3.1]{samei2007zero}, $ \Sp(R) $ does not have any quasi-isolated point, i.e., $ \Mi(R) $ does not have any isolated point.
	
	(c) $ \Rightarrow $ (a). It concludes from \cite[Corollary 4.3]{aliabad2013fixed}.
\end{proof}

\begin{Thm}
	The following statements are equivalent.
	\begin{itemize}
		\item[(a)] The zero ideal of $ R $ is an anti fixed-place ideal.
		\item[(b)] $ \AG(R) $ is triangulated.
		\item[(c)] $ \Mi(R) $ does not have any isolated point.
	\end{itemize}
\label{AG R is triangulated}
\end{Thm}
\begin{proof}
	(a) $ \Rightarrow $ (b). It is similar to proof of the part (a) $ \Rightarrow $ (b) of the previous theorem.
	
	(b) $ \Rightarrow $ (a). Suppose that the zero ideal of $ R $ is not an anti fixed-place ideal. Then $ P \in \mathcal{B}(R) \neq \emptyset $ exists, hence $ P $ is a affiliated prime ideal, so there is some $ a \in R $ such that $ \An(a) = P $, thus $ h_m^c(a) = \{ P \} $. This implies that $ \{ P \} $ is open in $ \Mi(R) $, therefore $ h_m^c(P) = \Mi(R) \setminus \{ P \} $ is closed and consequently $ \overline{h_m^c(P)} = \Mi(R) \setminus \{ P \} $. Thus $ h_m(P)^\circ = \left( \overline{h_m^c(P)} \right)^c = \{ P \}    $. Now Proposition \ref{triangle}, concludes that $ P $ is not a vertex of any triangle and therefore $ \AG(R) $ is not triangulated.
	
	(a) $ \Leftrightarrow $ (c). It is clear, by \cite[Corollary 4.3]{aliabad2013fixed}.
\end{proof}

In the \cite[Corollary 3.3]{samei2007zero}, it has been asserted that ``Let $R$ be a reduced ring and let $ \Sp(R)$ be finite. Then $ \Gamma(R)$ is a triangulated graph if and only if
$\Sp(R)$ has no isolated points.''. If $ \Sp(R) $ is finite, then $ \Mi(R) $ is finite, so the zero ideal of $ R $ is fixed-place and therefore it is not anti fixed-place, hence by the above theorem $ \Gamma(R) $ is not triangulated. Hence the assumption ``$ \Gamma(R)$ is a triangulated graph'' in this assertion is impossible.

Now we can conclude the following corollary from the above theorems.

\begin{Cor}
	$ \Gamma(R) $ is triangulated, \ff $ \AG(R) $ is triangulated.
\end{Cor}

Now we can conclude easily from Theorem \ref{Gamma R is triangulated} and \cite[Corollary 5.4]{aliabad2013fixed}, that $ \Gamma(X) $ is triangulated, \ff $ X $ does not have any isolated point. This fact has been shown in \cite[Proposition 2.1]{azarpanah2005zero}. Also, we can conclude easily from Theorem \ref{AG R is triangulated} and \cite[Corollary 5.4]{aliabad2013fixed}, that $ \AG(X) $ is triangulated, \ff $ X $ does not have any isolated point. This fact also has been shown in \cite[Theorem 4.5]{badie2020annihilating}. 

If $ \Mi(R) $ is finite, then the zero ideal of $ R $ is fixed-place and therefore it is not anti fixed-place, hence, by Corollary \ref{radius = 2}, $ \Ra \Gamma(R) = \Ra \AG(R) = 2 $.

Suppose that $ D $ is an integral domain and $ R $ be an arbitrary ring. Then $ \{0\} \times R \in \mathcal{B}( D \times R ) \neq \emptyset $, so the zero ideal of $ D \times R $ is not an anti fixed-place ideal, thus, by Corollary \ref{radius = 2} and Theorems \ref{Gamma R is triangulated} and \ref{AG R is triangulated}, $ \Ra \Gamma(R) = \Ra \AG(R) = 2 $ and the graphs $ \AG(R) $ and $ \Gamma(R) $ are not triangulated.

\section{Domination number}

The main purpose of this section is studying of domination number of $ \AG(R) $ and then $ \AG(X) $. In this studying, we employ the Bourbaki associated prime divisor of the zero ideal and the fixed-place ideal notion.

\begin{Lem}
	Let $ I $ be an ideal in $ \A{R} $. The following statements are equivalent.
	\begin{itemize}
		\item[(a)] $ I $ is prime.
		\item[(b)] $ I $ is a maximal element of $ \A{R} $.
		\item[(c)] $ I $ is a Bourbaki associated prime divisor of the zero ideal of $ R $.
	\end{itemize}
\label{element of B(R)}
\end{Lem}
\begin{proof}
	(a) $ \Rightarrow $ (b). Suppose that $ I \subseteq J $ and $ J \in \A{R} $, thus $ 0 \neq a \in \An(J) $ exists. Since $ R $ is a reduced ring, $ a \notin J $, then $ a \notin I $ and $ a J \subseteq I $, thus $ J \subseteq I $, hence $ I = J $. Consequently, $ I $ is a maximal element of $ \A{R} $.
	
	(b) $ \Rightarrow $ (c). Since $ I \in \A{R} $, there is some $ 0 \neq a \in R $ such that $ \An(a) =  I $. Suppose that $ xy \in I $ and $  x \notin I $, then $ I = \An(a) \subseteq \An(ax) $, so $ y \in \An(ax) \subseteq  \An(a) = I $, by the maximality of $ I $, hence $ I $ is prime, and therefore $ I $ is a Bourbaki associated prime divisor of the zero ideal. 
	
	(c) $ \Rightarrow $ (a). It is clear.
\end{proof}

\begin{Pro}
	The following statements hold.
	\begin{itemize}
		\item[(a)] Suppose that $ I \in \A{R} $. $ I $ is contained in some maximal element of $ \A{R} $, \ff $ \Mi(I) \cap \mathcal{B}(R) \neq \emptyset $.
		\item[(b)] Every element of $ \A{R} $ is contained in some maximal element of $ \A{R} $, \ff the zero ideal of $ R $ is a fixed-place ideal.
		\item[(c)] $ \A{R} $ does not have any maximal element, \ff the zero ideal of $ R $ is an anti fixed-place ideal.
	\end{itemize}	
\label{fixed-place and maximal}
\end{Pro}
\begin{proof}
	(a$ \Rightarrow $). By Lemma \ref{element of B(R)}, $ P \in \mathcal{B}(R) $ exists such that $ I \subseteq P $, since $ P \in \Mi(R) $, it follows that $ P \in \Mi(I) $ and therefore $ P \in \mathcal{B}(R) \cap \Mi(R) \neq \emptyset $.
	
	(a$ \Leftarrow $). It is clear, by Lemma \ref{element of B(R)}.
	
	(b$ \Rightarrow $). On contrary, suppose that $ \bigcap_{P \in \mathcal{B}(R)} P \neq \{ 0\} $, so there is some $ 0 \neq a \in \bigcap_{P \in \mathcal{B}(R)} P $. Then 
	\[ \An(a) = (0:a) = \big( \bigcap_{P \in \Mi(R)} P : a \big) = \bigcap_{a \notin P \in \Mi(R)} P  \]
	By the assumption, there is some $ P_\circ \in \mathcal{B}(R) $ such that $ \An(a) \subseteq P_\circ $, then $ \bigcap_{a \notin P \in \Mi(R)} P \subseteq P_\circ $, and therefore 
	\[\bigcap_{P_\circ \neq P \in \Mi(R)} P \subseteq \bigcap_{a \notin P \in \Mi(R)} P \subseteq P_\circ \] 
	\[ \quad \Rightarrow \quad \{ 0 \} = \Big( \bigcap_{P_\circ \neq P \in \Mi(R)} P \Big) \cap P_\circ = \bigcap_{P_\circ \neq P \in \Mi(R)} P \]
	which is a contradiction.
	
	(b$ \Leftarrow $). By the assumption, $ \bigcap_{P \in \mathcal{B}(R)} P = \{ 0 \} $. So
	\[ \An(I) = ( 0 : I ) = \big( \bigcap_{P \in \mathcal{B}(R)} P : I \big) = \bigcap_{P \in \mathcal{B}(R)} ( P : I ) = \bigcap_{I \not\subseteq P \in \mathcal{B}(R)} P \]
	Hence $ P \in \mathcal{B}(R) $ exists such that $  I \subseteq P $ and thus, by Lemma \ref{element of B(R)}, it completes the proof.
	
	(c). It is evident, by Lemma \ref{element of B(R)}.
\end{proof}

In the proof of \cite[Theorem 2.2]{nikandish2015domination} It has been asserted that ``By Zorn's Lemma, it is clear that if $ \A{R} \neq \emptyset $, then $ \A{R} $ has a maximal element''. But by the above proposition, we know that if the zero ideal of a ring $ R $ is anti fixed-place, then $ \A{R} $ does not have any maximal element. For example, since $ \R $ does not have any isolated point, by \cite[Corollary 5.4]{aliabad2013fixed}, the zero ideal of $ C(\R) $ is an anti fixed-place ideal and therefore $ \B(C(\R)) = \emptyset $. In this case, $ M = \B(C(\R)) = \emptyset $, so \cite[Theorem 2.2]{nikandish2015domination} is not true in general.

\begin{Thm}
	For each ring $ R $, 
	\begin{itemize}
		\item[(a)] $ |\B(R)|  \leqslant \dt_t(\AG(R)) $.
		\item[(b)] If $ | \Mi(R) | > 2 $, then $ |\B(R)| \leqslant \dt(\AG(R)) $.
	\end{itemize}
\label{B(R) < dt}
\end{Thm}
\begin{proof}
	(a). Suppose that $ D $ is a total dominating set of $ \AG(R) $. For each $ P \in \B(R) $, there is some $ I_P \in D $, such that $ I_P $ is adjacent to $ P $, so $ P I_P = \{ 0 \}  $, thus $ P \subseteq \An(I_P) $, hence $ P = \An(I_P) $, by Lemma \ref{element of B(R)}. Now suppose that $ I_P = I_Q $, for some $ P , Q \in \B(R) $, then $ P = \An(I_P) = \An(I_Q) = Q $ and thus the map $ P \rightsquigarrow I_P $ is one-to-one. This implies that $ |\B(R)| \leqslant |D| $ and consequently $ |\B(R)| \leqslant \dt_t(\AG(R)) $.
	
	(b). Let $ D $ be a dominating set. For each $ P \in \B(R) $, if $ P \in D  $, then we set $ K_P = P $ and if $ P \notin D $, there is some $ K_P \in D $ such that $ K_P $ is adjacent to $ P $. Suppose that $ K_P = K_Q $, for some $ P , Q \in \B(R) $. If $ P , Q \in D $, then $ P = K_P = K_Q = Q $. If $ P, Q \notin D $, then $ P $ and $ Q $ are adjacent to $ K_P $ and $ K_Q $, respectively, so $ P K_P = Q K_Q = \{ 0 \} $, thus $ P \subseteq \An(K_P) $ and $ Q \subseteq \An(K_Q) $ and therefore $ P = \An(K_P) = \An(K_Q) = Q $, by Lemma \ref{element of B(R)}. Finally, without loss of generality, we assume  $ P \in D $ and $ Q \notin D $, then $ P = K_P $ and $ K_Q $ is adjacent to $ Q $, so $ P  $ is adjacent to $ Q $ and thus $ P Q = \{ 0 \} $. Hence for each $ P' \in \Mi(R) $, $ P Q = \{ 0 \} \subseteq P' $, and therefore either $ P \subseteq P' $ or $ Q \subseteq P' $, so, by Lemma \ref{element of B(R)}, either $ P = P' $ or $ Q = P' $. This implies that $ | \Mi(R) | \leqslant 2 $, which contradicts the assumption. Consequently, the map $ P \rightsquigarrow K_P $ is one-to-one and thus $ |\B(R)| \leqslant \dt(\AG(R))$.
\end{proof}

\begin{Thm}
	If the zero ideal of $ R $ is a fixed-place ideal, then 
	\begin{itemize}
		\item[(a)] $ \dt_t(\AG(R)) = |\B(R)| $.
		\item[(b)] If $ | \Mi(R) | > 2 $, then $ \dt(\AG(R)) = |\B(R)| $.
	\end{itemize}
\label{fixed-place and dt}
\end{Thm}
\begin{proof}
	(a). By the above theorem it is sufficient to show that $ \dt_t(\AG(R)) \leqslant | \B(R) | $. For every $ P \in \B(R) $, pick $ a_{_P} \in R $, such that $ \An(a_{_P}) = P $. For each $ K \in \A{R} $, by the assumption and Proposition \ref{fixed-place and maximal}, there is some $ P \in \B(R) $ such that $ K \subseteq P = \An(a_{_P}) $, so $ Ra_{_P} K = \{ 0 \} $ and therefore $ K $ is adjacent to $ Ra_{_P} $. This implies that $ \{ Ra_{_P} : P \in \B(R) \} $ is a dominating set and consequently, $ \dt_t(\AG(R)) \leqslant | \B(R) | $.
	
	(b). By the fact that $ \dt(\AG(R)) \leqslant \dt_t(\AG(R)) $, it follows from (a) and the above theorem.
\end{proof}

We know that if $ \Mi(R) $ is finite, then the zero ideal of $ R $ is a fixed-place ideal and $  \Mi(R) = \B(R) $. Thus \cite[Theorem 2.4 and Theorem 2.5]{nikandish2014dominating} and \cite[Theorem 2.4 and Theorem 2.5]{nikandish2015domination} are immediate consequences of the above theorem. Also, we can conclude the following corollary from the above theorem and \cite[Theorems 5.2 and 5.5]{aliabad2013fixed}.

\begin{Cor}
	Suppose $ X  $ is an almost discrete space. Then 
	\begin{itemize}
		\item[(a)] $ \dt_t(\AGC) = | I(X) | $.
		\item[(b)] If $ | X | > 2 $, then $ \dt(\AGC) = | I(X) | $.
	\end{itemize}
\end{Cor}

\begin{Thm}
	If the zero ideal of a ring $ R $ is not a fixed-place ideal, then $ \dt(\AG(R)) $ and $ \dt_t(\AG(R)) $ are infinite.
\label{d is infinite}
\end{Thm}
\begin{proof}
	Suppose that $ D $ is a dominating set of $ \AG(R) $. By Proposition \ref{fixed-place and maximal}, there is some $ J_1 \in \A{R} $ which is not contained in a maximal element of $ \A{R} $. If $ J_1 \in D $, then we set $ I_1 = K_1 = J_1 $. If $ J_1 \notin D $, there is some vertex $ I_1 \in D $ which is adjacent to $ J_1 $, then $ J_1 I_1 = \{ 0 \} $, so $ J_1 \subseteq \An(I_1) $, in this case we set $ K_1 = \An(I_1) $. Since $ J_1 $ is not contained in a maximal element of $ \A{R} $ and $ J_1 \subseteq K_1 $, there is some $ J_2 \in \A{R} $ such that $ K_1 \subset J_2 $, similarly we can find $ K_2 \in \A{R} $ in which either $ I_2 = K_2 \in D $ or $ K_2 = \An(I_2) $, for some $ I_2 \in D $. By induction, we have the following
	\[ J_1 \subseteq K_1 \subset J_2 \subseteq K_2 \subset \ldots \subset J_n \subseteq K_n \subset \ldots \]
	Now suppose that $ n \neq m $, then $ K_n  \neq K_m $. Without loss of generality, we assume $ n < m $, hence we have four cases
	\begin{itemize}
		\item[case 1:] If $ I_n = K_n $ and $ I_m = K_m $, then it is evident that $ I_n \neq I_m $.
		\item[case 2:] If $ K_n = \An(I_n) $ and $ K_m = \An(I_m) $, so it is clear that $ I_n \neq I_m $.
		\item[case 3:] If $ K_n = I_n $ and $ K_m = \An(I_m) $, then $ I_n \subset \An(I_m) $, so $ I_n I_m = \{ 0 \} $, hence $ I_n \neq I_m  $, because otherwise, $ I_n^2 = \{ 0 \} $ and therefore $ I_n = \{ 0 \} $, which is a contradiction.
		\item[case 4:] If $ K_n = \An(I_n) $ and $ K_m = I_m $, then $ \An(I_n) \subset I_m $, so $ \An(I_m) \subseteq \An(\An(I_n)) $, hence $ I_n \neq I_m $, because otherwise, similar to case 3, $ \An(I_n) = \{ 0 \} $, which is a contradiction.
	\end{itemize}
	Since $ \{ I_n : n \in \N \} \subseteq D  $, it follows that $ D $ is infinite and consequently $ \dt(\AG(R)) $ is infinite. Hence $ \dt_t(\AG(R)) $ is finite, by this fact that $ \dt(\AG(R)) \leqslant \dt_t(\AG(R)) $.
\end{proof}

Now by the above theorem, $ \dt_t(\AG(C(\R))) $ and $ \dt(\AG(C(\R))) $ are infinite, so the inequality in Theorem \ref{B(R) < dt}, can be proper.

\begin{Cor}
	The following statements are equivalent
	\begin{itemize}
		\item[(a)] $ \dt_t(\AG(R)) $ is finite
		\item[(b)] $ \dt_t(\AG(R)) $ is finite
		\item[(c)] $ \Mi(R) $ is finite
	\end{itemize}
\end{Cor}
\begin{proof}
	It follows immediately from Theorems \ref{fixed-place and dt} and \ref{d is infinite} and this fact that if $ \Mi(R) $ is finite, then the zero ideal is a fixed-place ideal.
\end{proof}

Finally in the following proposition we generalize \cite[Theorem 2.3]{nikandish2015domination} to the infinite version.

\begin{Pro}
	For each reduced ring $ R $, we have $ \dt_t(\Gamma(R)) \leqslant \dt_t(\AG(R)) $.
\end{Pro}
\begin{proof}
	Suppose that $ D $ is a total dominating set of $ \dt_t(\AG(R)) $. So for each $ I‌ \in P $, there is some $ 0 \neq a_{_I} \in I $. For every $ a \in R $, there is some $ I ‌\in D $ such that $ I $ is adjacent to $ Ra $ in $ \AG(R) $, thus $ Ra I‌ = \{ 0 \} $, hence $ a a_{_I} = 0 $ and therefore $ a_{_I} $ is adjacent to $ a $ in $ \Gamma(R) $. Consequently, $ \{ a_{_I} : I‌ \in D \} $ is a  total dominating set of $ \Gamma(R) $ and this implies that $ \dt_t(\Gamma(R)) \leqslant \dt_t(\AG(R)) $.
\end{proof}

\end{document}